\documentclass[12pt]{interact}
\usepackage{epstopdf}
\usepackage{comment}
\usepackage{enumerate}
\usepackage[english]{babel}
\usepackage[colorlinks]{hyperref}
\usepackage{amsfonts,amsmath,amscd,amsthm,color}
\usepackage[all]{xy}
\usepackage[utf8]{inputenc}
\usepackage{amsmath}
\usepackage{amssymb}
\usepackage{amsfonts}
\usepackage{amsthm}
\usepackage{ graphicx}
\usepackage{fancyhdr}
\usepackage{mathdots}
\usepackage[linesnumbered,ruled]{algorithm2e}
\usepackage[T1]{fontenc}

\usepackage[caption=false]{subfig}

\usepackage[numbers,sort&compress]{natbib}
\bibpunct[, ]{[}{]}{,}{n}{,}{,}
\renewcommand\bibfont{\fontsize{10}{12}\selectfont}
\makeatletter
\def\NAT@def@citea{\def\@citea{\NAT@separator}}
\makeatother

\theoremstyle{plain}
\newtheorem{theorem}{Theorem}[section]
\newtheorem{lemma}[theorem]{Lemma}

\newtheorem{proposition}[theorem]{Proposition}

\theoremstyle{definition}
\newtheorem{definition}[theorem]{Definition}

\theoremstyle{remark}

\def\nR{\mathbb{R}}
\def\nC{\mathbf{C}}
\def\cB{\mathcal{B}}
\def\cG{\mathcal{G}}

\def\R{\textrm{R}}
\def\B{$\mathcal{B}_d$}

\def\para{\vskip0.25cm}

\newtheorem{ex}[theorem]{Example}

\newtheorem{obs}[theorem]{Remark}

\renewcommand{\theenumii}{\arabic{enumii}}
\makeatletter
\renewcommand{\p@enumii}{\theenumi.}
\makeatother

\numberwithin{equation}{section}

\begin{document}


\title{On  real Waring decompositions of real binary forms}

\author{
\name{M. Ansola\textsuperscript{a}\thanks{CONTACT: a) mansola@ucm.es; b) adiazcan@ucm.es; c) mangeles.zurro@uam.es }
, A. D\'{\i}az-Cano\textsuperscript{b} and M. A. Zurro\textsuperscript{c}}
\affil{
\textsuperscript{a}Universidad Complutense de Madrid, Madrid, Spain; \\
\textsuperscript{b}Universidad Complutense de Madrid. Facultad de Matem\'{a}ticas. IMI and Dpto. de \'{A}lgebra, Geometría y Topología, Madrid, Spain;
\thanks{Second author partially supported by Spanish MTM2014-55565 and Grupo UCM 910444.}
\thanks{Third author partially supported by Grupo UCM 910444.}
\\
\textsuperscript{c}Universidad Aut\'{o}noma de Madrid, Madrid, Spain}
}

\maketitle

\begin{abstract}
The Waring Problem over polynomial rings asks how to decompose a homogeneous polynomial $p$ of degree $d$ as a finite sum of $d$-{th} powers of linear forms. In this work  we give an algorithm  to obtain a real Waring decomposition of any given real binary form $p$ of length at most its degree. In fact, we construct a semialgebraic family of Waring decompositions for $p$. Some examples are shown to highlight the difference between the real and the complex case.
\end{abstract}


\vspace{-5mm}



  \section{Introduction}\label{sec-Intro}

{Among the many questions raised in the correspondence of Euler-Goldbach, in the middle of the  $18^{th}$ century, we find what is known as "Euler's four-square identity". This identity was used by Lagrange to prove in 1770 his "four square theorem". That same year, E. Waring  proposed} in his work {\it Meditationes Algebraicae} (see the english translation \cite{Wa}) as a conjecture that every positive integer could be expressed as a sum of no more than 9 cubes, or as a sum of at most 19 fourth powers. It was D.  Hilbert in 1909 who definitely proved 
 that every positive integer is a finite sum of $n$-th powers of positive integers, where the number of summands depends on $n$ (see  \cite{Boy}). This classical Waring Problem can be extended to polynomial decompositions in this way: any homogeneous polynomial $p$ of degree $d$ in $n$ variables over a field can be written as the sum of  at most $d$-{th} powers of linear forms.

\para 

The existence of this type of polynomial decompositions  is a particular case of the general problem of symmetric tensor decompositions as a sum of rank one tensors. This problem is of great interest due to the numerous applications not only in Mathe\-ma\-tics, but also in  Engineering (see \cite{BCM} and the references therein), Algebraic {Statistics}, Biology, Data {Mining}, Machine l\textcolor{blue}{L}earning, \cite{CGLM}, \cite{KB}, \cite{La}, and  Theoretical  Physics, cf. \cite{BC}.

\para

From an algebraic point of view, there are many open problems. For example, in 2018, B. Shapiro et al., \cite{FLOS}, have proposed a very interesting list of open questions on these issues when the field of coefficients is $ \mathbb{C}$. {One of the longest studied aspects is to determine the length of this decompositions or find lower bounds for them. Since J. Alexander and A. Hirschowitz, \cite{AH}, numerous authors have developed this problem (see \cite{Ble} and the references therein). The relation between the number of real linear factors and the real Waring rank of binary forms has also been studied by several authors (see \cite{CO}, \cite{Tok}  or \cite{BM}). Finally, we can point out some works more focused on the algorithmic aspects such as  \cite{BFPT1}, \cite{BFPT2}, \cite{GMKP} or \cite{P}, for pointing out some of the most recent. }
\para

In this paper we present {a procedure} to compute a Waring decomposition for a given  binary form $ p (x, y) $ with real coefficients. Observe that, if we fix a degree $ d> 1 $, the problem of  rewriting  $ p $ as a sum of powers of linear forms can be formulated as follows.

\begin{minipage}[c]{0.12\textwidth}
$$
(WD)_m
$$
\end{minipage} \hfill
\begin{minipage}[t]{0.7\textwidth}
Let  $p(x,y)$  be a binary form of degree $d$ in $\mathbb{R}[x,y]$.  Determine linear forms $\ell_1 ,\dots ,\ell_m $ in  $\mathbb{R}[x,y]$ such that $p(x,y)= c_1 \ell_1^d + \cdots +c_m \ell_m^d$ for some $c_1 , \dots , c_m$ real numbers.
\end{minipage} \hfill
\begin{minipage}[c]{0.1\textwidth}
    \begin{equation}\label{eq-problem}
    \quad
    \end{equation}
\end{minipage}
\bigskip

The number $m$ in $(WD)_m  $ is called {\it the length} of the decomposition. This decomposition of $p$ (not necessa\-rily unique) is known as {\it a Waring decomposition} (expression that we will abbreviate from now on as WD) of the polynomial. 
 
 \para 

This family of problems is solved trivially for $ m = d + 1 $. It is known (see \cite{CO}) that this decomposition is always possible for $ m = d $. As A. Causa and R. Re in \cite{CR}, or E. Carlini et al. in \cite{CCG}, affirm, the real case becomes more complicated than the complex case. Also \cite{BB} emphasizes the importance of the real case for the applications. This real binary case has been recently investigated by different authors (for instance, \cite{BCG}, \cite{CO} or \cite{Re}).  We study similar problems  to those described in \cite{FLOS} over the complex field, although the techniques necessary to address them in the real case rely on  semialgebraic geometry techniques associated with the corresponding projective space of real binary forms of fixed degree. Other authors have also used semialgebraic geometry to address pro\-blems with tensors (see the recent article of P. Comon et al. \cite{CLQY}). 

\para

One of our contributions to the problem $(WD)_d $ is an algorithm, Algorithm \ref{alg-WD-new}, to compute  WD for $p$ of length at most $d$. For complex binary forms similar problems have been studied in \cite{BFPT1},  \cite{BFPT2} using different techniques. We include {an illustrative case} to compare the real and the complex approaches  in  Ex. \ref{ex-complex-real}.

\para

 The work we present focuses on the study of real binary forms  with the ultimate goal of obtaining  WD by computational routines. Moreover, our techniques give a semialgebraic family of such decompositions, (see Section \ref{sec-WD} and theo\-rems therein).

\para

{\it The paper is organized as follows}. In Section \ref{sec-WD},  Theorems \ref{thm-odd} and \ref{thm-even}  present the parametric behavior of  Waring decompositions in $ (WD)_d $ for a real binary form $p$ depending on the parity of its degree $ d $. We would like to point out that, on the real field, parity has important topological implications (see the work \cite{CLQY} on these issues) and a unified treatment does not seem natural, to our understanding. 
\para

Section \ref{algoritmo} includes an algorithm for computing such  WD  (see Algorithm \ref{alg-WD-new}). It is based on an effective method to choose an appropriate set of parameters (see Algorithm \ref{alg-eleccion-Si-impar} and Algorithm \ref{alg-eleccion-Si-par}). The correctness of the algorithms {is}  proved in Propositions \ref{correct-odd}, \ref{correct-even}, and Theorem  \ref{thm-correctness}. Illustrative examples are included in Section \ref{sec-example}.  We have used Maple 18 to perform these computations.



\section{Real Waring decomposition of length at most $d$}\label{sec-WD}

Let   \B $\;$be the real  vector space of real binary forms of degree $d$ in the variables $x,y$.
Let  $p(x,y)$ be a real binary form in \B,
\begin{equation}\label{defp}
p(x,y)=p_{\vec{c}}(x,y)= \sum_{i=0}^{d}                         
\left(
\begin{minipage}[c][9pt][b]{8pt}
$$  \begin{array}{c}
   \vspace{-4pt}
   \!\! \scriptstyle d\!\! \\
   \vspace{-4pt}
    \!\! \scriptstyle i  \!\!
  \end{array} $$ 
\end{minipage}
\right)
c_{i}\, x^{i}\, y^{d-i}, \quad \text{with } \vec{c}=(c_0,\ldots, c_d) \in \mathbb{R}^{d+1}\setminus\{\vec{0}\}
\end{equation}

A {\sl Waring Decomposition (WD) over $\mathbb{R}$ of length $r$ for $p$}  is any rewrite of the form $p$ as a linear combination  of $d$-th powers of linear forms $\ell_i =\alpha_i x+\beta_i y$, $i=1,\dots ,r$, say
\begin{equation}\label{desWaring}
  p(x,y)=\sum_{i=1}^{r}\lambda_i\ell_i^d \ , \quad \text{for some real numbers } \lambda_i \, .
\end{equation} \label{eq:WD}
We also require that this expression is not redundant, that is, $\ell_1 ,\dots ,\ell_r $ are pairwise non-proportional. The number $r$ is called {\sl the length of the WD}.  Moreover,  when the  expression \eqref{desWaring} has minimal length, $ r $ is called the {\sl real rank of $p$} .

\para

Let us fix a real binary form $p$ en $\cB_d$. In this section we present a procedure to compute a WD for $p$ with length at most $d$. This number is an upper bound for the real rank of $p$. This was proved in  \cite{CO}, Prop. 2.1, but { no explicit constructions were} given there. This  upper bound does not seem as sharp as Theorem 1.1. {in \cite{Re3},}  where the length of the WD for a {(complex)} binary form is bounded by $\ \frac{d+1}{2}\ $ or $\ \frac{d}{2}+1\ $, depending on whether $\ d\ $ is odd or even respectively. But it is important to notice that our statement refers to "any {real binary form}" while Sylvester talks about "a {(complex)} general binary form". This "genericity condition" fails, for instance, for the form $p(x,y)=ax^3+3bx^2y-3ax^2y-by^3$.  It is easy to check that can not be written as a sum of less than {three real} linear forms, while, according to the previous limit, the expected length should be $2$. This form will be studied in more detail in Example \ref{ex-complex-real}.

\para 

{Next we give a semialgebraic description of a family of WD for any real binary form. We will establish our results according to the parity of d. The proof is constructive, which provides a method to construct the WD that is stated}.

\para 
{
 \noindent\textbf{Notation.} We  use the following notation along this work: for a matrix $M$, the expression 
${M}^t$  denote its transpose. In particular, we use $\vec{\lambda}^t$ to refer to the column vector corresponding to the row vector $\vec{\lambda}$.
 }

\vspace{1cm}

\noindent \textbf{\textsc{Construction for  odd degrees.}}
\para

In this section we give a parametric presentation of the WD of a given real binary form $p$ of odd degree $d$. The parameter space is an open semialgebraic set $\mathcal{G}$ in $\mathbb{R}^{(d-1)/2}$. We can establish the following result.

\begin{theorem}\label{thm-odd}
Let $p(x,y)=p_{\vec{c}}(x,y)= \sum_{i=0}^{d}  \binom{d}{i}    c_{i}\, x^{i}\, y^{d-i}$ be a real binary form  of odd degree $d=2\nu+1$. Then, there exists an open and dense semialgebraic set $ \mathcal{G}$ in $\mathbb{R}^\nu$ such that $p$ has the following WD:
\begin{equation} \label{eq:WD_teo_impar}
p(x,y) = \sum_{j=1}^{\nu} \lambda_{2j-1} \,(x+s_j\, y)^{d} \, + \,  \sum_{j=1}^{\nu} \lambda_{2j} \,(x-s_j\, y)^{d} \, + \, \lambda_d (R_1 x +R_2 y )^d \; ,    
\end{equation}
with  $s_j , -s_j $ in $  \mathcal{G}$ for all $j$ and $R_1 , R_2$ rational expressions  in $(s_1 , \dots , s_\nu )$  in $ \mathbb{R}^{\nu } $. 

\end{theorem}

\begin{proof}
Assume that $d=2\nu+1$.  Take a nonzero real binary form $p(x,y)$ as in \eqref{defp}, and  $\vec{c} = (c_0 ,\dots, c_d )$ the point of $\mathbb{R}^{d+1}\backslash\{\vec{0}\}$ associated to $p(x,y)$.
In fact, running Algorithm \ref{alg-WD-new} over "a symbolic ring", we obtain a one-parameter WD.

We consider two independent sets of  variables over  $\mathbb{R}$, say  $\{X_0 ,\dots ,X_d\}$ and $ \{S_1  ,\dots ,S_\nu\}$,  for  $\nu=(d-1)/2$, and  the $(d+1)\times (d+1)$ matrix

\begin{equation} \label{V_impar}
    V=
\left(\begin{array}{ccccccc}  X_{0}&1&1&\cdots&1&1&c_{d} \\ X_{1}&S_{1}&-S_{1}&\cdots&S_{\nu}&-S_{\nu}&c_{d-1} \\
\vdots&\vdots&\vdots&\cdots&\vdots&\vdots&\vdots \\
X_{d}&S_{1}^{d}&(-1)^{d}S_{1}^{d}&\cdots&S_{\nu}^{d} &(-1)^{d}S_{\nu}^{d}&c_{0}\end{array}\right) \, .
\end{equation} 
Its determinant is a linear polynomial in $X_0 , \dots , X_d$ with coefficients $\Delta_0 , \dots , \Delta_d$ in the ring $\mathbb{R} [S_1,\cdots, S_{\nu}]$, say 
\begin{equation}\label{eq-det-odd}
\det (V)= h(X_0 ,\dots ,X_d )=\Delta_0 X_0 \!+\!\Delta_1 X_1 \!+\!\cdots +\! \Delta_d X_d .    
\end{equation}
First, let us assume that $\Delta_d$ is a nonzero polynomial. Then,
$\Omega = \mathbb{R}^\nu \setminus \{\Delta_d = 0 \}$ is  open and dense in $\mathbb{R}^\nu$ {(see the section {\it Algebraic and Semi-algebraic Sets} in \cite {BCR})} and  the quotient $R =-\Delta_{d-1}/\Delta_d  $ defines a rational function in $\Omega$. In order to find the suitable linear forms to solve $(WD)_d$ for $p$, see \eqref{eq-problem}, we will restrict this open set to assure the $\lambda_i$ in \eqref{desWaring} exist, $1\leq i \leq d$. In fact,  we consider the real algebraic sets in $\mathbb{R}^\nu$
$$
A= \bigcup_{i=1}^{\nu} \{S_i=0\}  ,\quad D= \left(\bigcup_{i<j} \{S_i + S_j=0\}\right) \cup \left(\bigcup_{i<j} \{S_i - S_j=0\}\right),
$$
$$
B= \left(\bigcup_{i=1}^{\nu} \{\Delta_{d-1}+ S_i\Delta_d=0\}\right) \cup \left(\bigcup_{i=1}^{\nu} \{\Delta_{d-1}- S_i\Delta_d=0\}\right)
$$
and then 
\begin{equation}\label{conj-G-odd}
 \mathcal{G} =\Omega\setminus (A\cup B\cup D)  
\end{equation}
is an open  semialgebraic set in $\mathbb{R}^\nu$. Moreover, the open semialgebraic set  $\mathcal{G}$ is non empty since $\Omega$ is dense, and we can choose a point $\textsc{s}=(s_1, \cdots, s_{\nu}) \in \mathcal{G}$. Then, the real polynomial
$$
h^* (T)=h(1,T,T^2 ,\dots ,T^d )=\Delta_0 (\textsc{s}) +\Delta_1 (\textsc{s})T +\cdots +\Delta_d(\textsc{s}) T^d  \
$$
has $d=2\nu+1$ real roots: $\pm s_i \in \mathbb{R}\setminus{\{0\}}$ and also $\R=R(\textsc{s} )$, which are distinct  because $\textsc{s}\in\mathcal{G}$.
\para 
Now, we consider the matrix
\begin{equation} \label{eq:M_impar}
M=
\left(
\begin{array}{cccccc}  1&1&\cdots&1&1 &1\\ s_{1}&-s_{1}&\cdots&s_{\nu}&-s_{\nu} &R\\
\vdots&\vdots&\cdots&\vdots&\vdots &\vdots\\
s_{1}^{d}&(-1)^{d}s_{1}^{d}&\cdots&s_{\nu}^{d}&(-1)^{d}s_{\nu}^{d} &R^{d}\end{array}
\right) \ ,
\end{equation} 
and   the linear system:
\begin{equation}\label{eq:sistema1}
    M\vec{\lambda} =\bar{c} 
\end{equation}
where ${\bar{c}} =(c_d ,\dots , c_0)^{\,\,t}$ and  ${\vec{\lambda}} =(\lambda_1\, ,\dots ,\,\lambda_d)^{\,t}$. 
We point out that $M$ is a   $(d+1) \times \,d$ matrix of  rank  $d$, by the election of \textsc{s}. But also the determinant  $\det (\,M|\,\bar{c}\,)$  equals zero,  since we have $h^{*}(1,\textrm{R},\textrm{R}^2,\,\cdots,\,\textrm{R}^d)=0$.  Thus, the system \eqref{eq:sistema1} can be solved to obtain $\vec{\lambda}^* =( \lambda_1^* , \dots , \lambda_d^* )^t$.  Then, we verify the statement in this case with   $ \lambda_j = \lambda_{j}^*$, for $j=1, \dots , d$, and $(R_1 , R_2 )= (1, \R)$. Hence,

\begin{equation}\label{eq:WD_impar}
p(x,y) = \sum_{j=1}^{d} \lambda_j^* \,L_{j}^{d}\, (x,y) ,
\end{equation}

\noindent with $L_j (x,y)=x+s_{\frac{j+1}{2}}\, y$, if $j$ is odd,  $L_j (x,y)=x-s_\frac{j}{2} \,y$, if $j$ is even, when $j<d$, and $L_d (x,y)=x+\R y$. In this way the theorem is proved in this case.

\medskip

For the remaining cases, that is for those $p$ such that $\Delta_d \equiv 0$ for all $\textsc{s}\in \mathcal{G}$, we proceed as follows.
\para 

Choose a point $\textsc{s}\in \mathcal{G}$,  and consider the matrix 
\begin{equation}\label{eq:M_impar_0}
   M_0=
\left(
\begin{array}{cccccc}  1&1&\cdots&1&1 &0\\ s_{1}&-s_{1}&\cdots&s_{\nu}&-s_{\nu} &0\\
\vdots&\vdots&\cdots&\vdots&\vdots &\vdots\\
s_{1}^{d}&(-1)^{d}s_{1}^{d}&\cdots&s_{\nu}^{d}&(-1)^{d}s_{\nu}^{d} &1\end{array}
\right)  \ .
\end{equation}
Now the corresponding linear system is 
\begin{equation}\label{eq:d-impar-0}
M_0\vec{\lambda} =\bar{c},
\end{equation}
with ${\bar{c}} $ and  ${\vec{\lambda}}$ as above. Again,  $M_0$ is a   $(d+1) \times \,d$ matrix of  rank  $d$ by construction. Moreover the determinant  $\det (\,M|\,\bar{c}\,)$  equals zero, because we have $\Delta_d \equiv 0$.  Thus, the system \eqref{eq:d-impar-0} can be solved to obtain $\vec{\lambda}^* =( \lambda_1^* , \dots , \lambda_d^* )^t$.  Then, we can write $p$ as in \eqref{eq:WD_teo_impar}, taking  $ \lambda_j = \lambda_{j}^*,\,$  for $j=1, \dots , d$, and $(R_1 , R_2 )= (0,1)$, as we wanted to prove.
\end{proof}

Next, we give an example of the previous procedure.

\begin{ex}  Take $p(x,y)= 5226y^5+4970xy^4+1860x^2y^3+340x^3y^2+30x^4y+x^5$. \label{exp-p5} Choosing $\textsc{s}=(1,\,2)$, we obtain $R=\frac{120}{23}$ and we have a WD  with length $5$,  but if $\textsc{s}=(3,\,4)$ then $R=5$ and we have found a WD  with length $2$, value that, in this case, gives us the real rank of $p$. Hence, we can write
$$
\begin{array}{rl}
{\scriptstyle p(x,y)\,}=&\!\!\! \frac{70}{97}{\scriptstyle(x+y)^5} - \frac{28}{143}{\scriptstyle(x-y)^5} - \frac{35}{37}{\scriptstyle(x+2y)^5} + \frac{5}{83}{\scriptstyle(x-2y)^5} + \frac{57927087}{42597841} {\scriptstyle\left(x+\frac{120}{23}\right)^5}\\
\, \\
  =&\!\!\!-(x+4y)^5+2(x+5y)^5.
\end{array}
$$

\end{ex}

\begin{obs}\label{raro-d3}
  Observe that for $p(x,y)=3x^2 y + y^3$, we have $\Delta_3=\left|\begin{array}{ccc}
       1&1&0  \\
       s_1&-s_1&1 \\
       s_1^2&s_1^2&0 
  \end{array}\right|=0$, for any choice of $s_1$. So, we construct $M_0$ as in \eqref{eq:M_impar_0} and the system \eqref{eq:d-impar-0} gives
  $$
  \lambda_1 = \frac{1}{2 s_1} ,\   \lambda_2 = -\frac{1}{2 s_1} ,\  \lambda_3 = 1-s_1^2.
  $$
 for every $s_1 \not =0$. Then, 
 $$
 3x^2 y + y^3 = \frac{1}{2 s_1 } \left( x+s_1 y \right)^3 -
 \frac{1}{2 s_1} \left( x-s_1 y \right)^3 +( 1-s_1^2) y^3
 $$
 \noindent and we get a one-parameter WD.
 
 \smallskip
 
Note that if $s_1=\pm 1$ we obtain a shorter expression that shows us that $p$ is a real binary form of real rank $2$.
\end{obs}

\para 

\noindent \textbf{\textsc{Construction for  even degrees.}}
\para{}
Now, we analyze the WD of length at most $d$ for a given real binary form $p$ of even degree $d$. We give a parametric presentation  where the parameter space is an open semialgebraic set $\mathcal{G}$ in $\mathbb{R}^{d/2}$. We can establish the following result.

\begin{theorem}\label{thm-even}
Let $p(x,y)= \sum_{i=0}^{d}  \binom{d}{i}    c_{i}\, x^{i}\, y^{d-i}$ be a real binary form  of even degree $d=2\nu$. Then, there exists an open and dense semialgebraic set $ \mathcal{G}$ in $\mathbb{R}\times\mathbb{R}^{\nu -1 }$ such that $p$ has the following WD:
\begin{equation}\label{eq:WD_teo_par}
p(x,y) = \sum_{j=1}^{\nu-1} \lambda_{2j-1} (x+s_j\, y)^{d}  +   \sum_{j=1}^{\nu-1} \lambda_{2j} (x-s_j\, y)^{d}  +  \lambda_{d-1} ( x +s y )^d    + \lambda_d (R_1 x +R_2 y )^d   
\end{equation}
with  $s, s_j , -s_j $ in $  \mathcal{G}$ for all $j$ and $R_1 , R_2$ rational expressions of $(s,s_1 , \dots , s_{\nu-1})$ in $  \mathbb{R} \times \mathbb{R}^{\nu -1} $ .

\end{theorem}

\begin{proof}
Let  $p(x,y)$ be a nonzero real binary form  as in \eqref{defp} of degree $d=2\nu$ and  $\vec{c} = (c_0 ,\dots, c_d )$ the vector of $\mathbb{R}^{d+1}\backslash\{\vec{0}\}$ associated to $p(x,y)$.

\para

Now we consider two independent sets of variables over  $\mathbb{R}$, say  $\{X_0 ,\dots ,X_d\}$ and $\{S, S_1  ,\dots ,S_{\nu -1}\}$. In this case we construct  the $(d+1)\times (d+1)$ matrix:

\begin{equation}\label{eq:V_par}
V=
\left(\begin{array}{cccccccc}  X_{0}&1&1&1&\cdots&1&1&c_{d} \\ X_{1}&S&S_{1}&-S_{1}&\cdots&S_{\nu-1}&-S_{\nu-1}&c_{d-1} \\
\vdots&\vdots&\vdots&\vdots&\cdots&\vdots&\vdots&\vdots \\
X_{d}&S^d &S_{1}^{d}&(-1)^{d}S_{1}^{d}&\cdots&S_{\nu-1}^{d} &(-1)^{d}S_{\nu-1}^{d}&c_{0}\end{array}\right) \ .
\end{equation} 
Its determinant is a linear polynomial in $X_0 , \dots , X_d$ with coefficients $\Delta_0 , \dots , \Delta_d$ in the ring $\mathbb{R} [S,S_1,\cdots, S_{\nu-1}]$, say 
\begin{equation}\label{eq-det-even}
\det (V)= h(X_0 ,\dots ,X_d )=\Delta_0 X_0 \!+\!\Delta_1 X_1 \!+\!\cdots +\! \Delta_d X_d .    
\end{equation}
Now, let us consider the polynomial $\Delta_{d}$. First assume  $\Delta_{d}\not\equiv 0$.  Thus,  $\Omega=\mathbb{R}^{\nu}\setminus \{ \Delta_d =0\, \}$ is an open and dense semialgebraic set.  {We also} consider the real algebraic sets in $\mathbb{R}^\nu$:
$$
A= \bigcup_{i=1}^{\nu-1} \{S_i=0\},\quad 
B= \left\{
\, \Delta_{d-1}+2 S\Delta_d=0  \ \right\} 
\cup 
\left(\bigcup_{i=1}^{\nu-1} \{\Delta_{d-1}+ (S \pm S_i ) \Delta_d=0\}\right),
$$

$$
 D=\left(\bigcup_{i<j} \{S_i \pm S_j=0\}\right) \cup 
\left( \bigcup_{i=1}^{\nu-1} \{S \pm S_i=0\} \right) .
$$
Then 
\begin{equation}\label{conj-G-even}
\mathcal{G} =\Omega\setminus (A\cup B\cup D)
\end{equation}
is an open  semialgebraic set in $\mathbb{R}^\nu$. Moreover, since $\mathcal{G}$ is non empty,  we can choose $\textsc{s}=(s^* ,s_1, \cdots, s_{\nu-1}) \in \mathcal{G}$. Thus the real polynomial
$$
h^* (T)=h(1,T,T^2 ,\dots ,T^d )=\Delta_0 (\textsc{s}) +\Delta_1 (\textsc{s})T +\cdots +\Delta_d(\textsc{s}) T^d  \
$$
has $d=2\nu$ real roots: $\pm s_i \in \mathbb{R}\setminus{0}$, and also  $s^*$ and  $R=-\dfrac{\Delta_{d-1}(\textsc{s})}{\Delta_d (\textsc{s})}-s^*$, which are distinct  because $\textsc{s}\in \mathcal{G}$.

In this case,  we consider the $(d+1)\times d$ matrix
\begin{equation}\label{eq:M_par}
M=
\left(\begin{array}{cccccccc} 1&1&1&\cdots&1&1&1\\ s&s_{1}&-s_{1}&\cdots&s_{\nu-1}&-s_{\nu-1} &R\\
\vdots&\vdots&\vdots&\cdots&\vdots&\vdots &\vdots\\
s^d&s_{1}^{d}&(-1)^{d}s_{1}^{d}&\cdots&s_{\nu-1}^{d}& (-1)^{d}s_{\nu-1}^{d}&R^d\end{array}\right)
\end{equation} 
and the corresponding linear system
\begin{equation}\label{eq:sistema2}
    M\vec{\lambda} =\bar{c}
\end{equation}
with ${\bar{c}} =(c_d ,\dots , c_0)^{\,\,t}$ and  ${\vec{\lambda}} =(\lambda_1\, ,\dots ,\,\lambda_d)^{\,t}$. 
Just like in the odd case, $M$ is a   $(d+1) \times \,d$ matrix of  rank  $d$ and also the determinant  $\det (\,M|\,\bar{c}\,)$  equals zero.  Thus, the system \eqref{eq:sistema2} can be solved to obtain $\vec{\lambda}^* =( \lambda_1^* , \dots , \lambda_d^* )^t$. Hence,

\begin{equation}\label{eq:WD_par}
p(x,y) = \sum_{j=1}^{d-2} \lambda_{j+1}^* \,L_{j}^{d}\, (x,y) + \lambda_1^*   L_{d-1}^d (x,y) +\lambda_d^*  L_d^d (x,y) \ ,
\end{equation}
\noindent
where, for $1\leq j<d-1$, we write  $L_j (x,y)=x+s_{\frac{j+1}{2}}\, y$, if $j$ is odd,  $L_j (x,y)=x-s_{\frac{j}{2}}\, y$, if $j$ is even,  $L_{d-1} (x,y)=x+s^* y\,$  and  $\,L_d (x,y)=x+\R y$. Then, defining  $ s=s^*,\ \lambda_d = \lambda_d^*$, $ \lambda_{d-1} = \lambda_1^*$, {and} $ \lambda_j = \lambda_{j+1}^*$, for $j=1, \dots , d-2$,  {and} $(R_1 , R_2 )= (1, \R)$, we have proved the statement in this case.

\medskip
Next, let    $\ \Delta_d $ be the zero polynomial. Now choose a point $\textsc{s}\in \mathcal{G}$,  and consider the $(d+1)\times d$ matrix 
\begin{equation}\label{eq:M_par_0}
M_0=
\left(\begin{array}{cccccccc} 
1&1&1&\cdots&1&1&0
\\ s&s_{1}&-s_{1}&\cdots&s_{\nu-1}&-s_{\nu-1} &0\\
\vdots&\vdots&\vdots&\cdots&\vdots&\vdots &\vdots\\
s^d&s_{1}^{d}&(-1)^{d}s_{1}^{d}&\cdots&s_{\nu-1}^{d}& (-1)^{d}s_{\nu-1}^{d}&1\end{array}\right)    \ .
\end{equation}
The corresponding linear system is
\begin{equation}\label{eq:sistema-par-0}
    M_0\vec{\lambda} =\bar{c} 
\end{equation}
with ${\bar{c}}\,$ and  $\,{\vec{\lambda}} \,$ as before. 
Again, $M_0$ is a   $(d+1) \times \,d$ matrix of  rank  $d$, by construction, and  also the determinant  $\det (\,M|\,\bar{c}\,)$  equals zero.  Thus, the system \eqref{eq:sistema-par-0} can be solved to obtain $\vec{\lambda}^* =( \lambda_1^* , \dots , \lambda_d^* )^t$. Hence,
\begin{equation}\label{eq:WD_par_0}
p(x,y) = \sum_{j=1}^{d-2} \lambda_{j+1}^* \,L_{j}^{d}\, (x,y) + \lambda_1^*  L_{d-1}^d (x,y) +\lambda_d^* \, L_d^d \ ,
\end{equation}
\noindent
where, for $1\leq j<d-1$, we write  $L_j (x,y)=x+s_{\frac{j+1}{2}}\, y$, if $j$ is odd,  $L_j (x,y)=x-s_{\frac{j}{2}}\, y$, if $j$ is even, and also  $L_{d-1} (x,y)=x+s^* y$,  $L_d (x,y)=y$. Then, defining  $s=s^*$, $ \lambda_d = \lambda_d^*$, $ \lambda_{d-1} = \lambda_1^*$ and $ \lambda_j = \lambda_{j+1}^*$ for $j=1, \dots , d-2$, and $(R_1 , R_2 )= (0,1)$,  {we obtain  the proof in this case}.
\end{proof}

\medskip

Next, we compute an example using the previous procedure.

\begin{ex}\label{ex-4-2}  Take $p(x,y)= 240y^4+224xy^3+72x^2y^2+8x^3y+x^4$. \label{exp-p4}  We can choose $\textsc{s} =(0,1)$ and then $R=\frac{38}{9}$. But, for  $\textsc{s} =(0,2)$ we obtain $R=4$. Hence, we get  two decompositions of length $4$ and $2$ (the one that gives us the real rank), respectively:
$$
\begin{array}{rl}
p(x,y)&=\displaystyle \frac{34}{19}x^4-\frac{40}{29}(x+y)^4-\frac{8}{47}(x-y)^4 +\frac{19683}{25897}\left(x+\frac{38}{9}y\right)^4=\\
\, \\
  &=-(x+2y)^4+(x+4y)^4.
\end{array}
$$
\end{ex}

\medskip

\begin{obs} Consider the following analytic path 
\[
\gamma : \ [ 0 , 1] \rightarrow \cB_d \,, \ 
\gamma (\varepsilon ) = (\varepsilon^2 +1) y^4+6 \, \varepsilon^2 x^2 y^2   + 4\varepsilon x^3 y \ .
\]
It joints the binary forms $\ \gamma(0)= y^4\,$ and $\ \gamma (1)= 2y^4+6 \,  x^2 y^2   + 4 x^3 y$. Next we apply our construction to the form $\gamma (\varepsilon ) $. Observe that  the system associated with the matrix \eqref{eq:M_par},  for  $s_1 =1$ {and any value of $s^*$}, gives $\,R=-1/\varepsilon\,$ and
\[
  \lambda_1 =0 , \quad \lambda_2 = \dfrac{\varepsilon(\varepsilon^2 + \varepsilon + 1)}{2(\varepsilon + 1)}, \quad \lambda_3 =\dfrac{\varepsilon(\varepsilon^2 - \varepsilon + 1)}{2(\varepsilon - 1)} \ ,\quad  \lambda_4=-\dfrac{\varepsilon^4 }{\varepsilon^2 - 1}
\]
and then, for $\varepsilon\in [0,1)$, we have 
 $$
\gamma (\varepsilon )= \lambda_2  \left( x+ y \right)^4 + \lambda_3 \left( x- y \right)^4 
 -\dfrac{1 }{\varepsilon^2 - 1}  \left( \varepsilon x- y \right)^4 
 $$
Therefore, forms with WD of length $1$ can be in the closure of the set of forms with WD of length $3$. 
\end{obs}

\begin{obs}\label{remark-Gp}
Observe that real Waring decompositions for binary forms provide a morphism of real algebraic sets:
\begin{align*}
 \xi: &\hskip1.5cm  \mathbb{P}^d &\longrightarrow\qquad &\hskip2cm  \mathbb{P}^d =\mathbb{P}(\cB_d   )\\
& \underbrace{ \left[\,a_1,\,b_1, \cdots , a_\nu,\,b_\nu\,R_1,R_2\right]}_{2\nu+2\,=\,d+1} &\longrightarrow \qquad
& \sum_{i=1}^{\nu}
(a_i x \pm b_i y)^d+  (R_1\,x+R_2\,y)^d \ .
\end{align*}

\noindent We point out that for $d$ an odd number, Theorem \ref{thm-odd} guarantees that $\xi$ is a surjection. 

The analogous mapping for even degrees is no longer surjective, but nevertheless Theorem \ref{thm-even} gives the description of a family of WD for $p$. It is an interesting question to know if there is any description of the fiber for the case of even degree that takes into account the possible signs that affect the needed linear forms   to build a WD. Similar problem is treated in the work of B. Reznick \cite{Re}.

\end{obs}

\section{The algorithm \texttt{Real Waring Decomposition} (RWD) }\label{algoritmo}

\medskip
In this section we present an algorithm to compute a real WD of length at most $d$ for any real binary form. Its correctness is based in some easy result that we quoted for the convenience of the reader. The idea of the algorithm is to guaranty an effective choice of an element of the set $\cG$ in Theorem \ref{thm-odd} or Theorem \ref{thm-even}. When the $s_i$ and $s$ in those theorems are considered as parameters, an appropriate use of Lemma \ref{lema-roots} allows to compute the desired WD and hence to solve effectively $(WD)_d$.

\para

Clearly, given a real binary form one can consider it as a complex binary form and try to use the recent algorithms to find its Waring decomposition where now the $ \lambda_i $ would be, in principle, complex numbers. Observe that if we apply the Sylvester's Algorithm to $p$,  there is no guarantee that the linear forms we obtain have real coefficients. This fact is quite delicate and it relies on the fact that $\mathbb{R}$ is a real closed field  in an essential way. We will review this method before exposing our algorithm. We would like to point out that the linear forms we give for decomposing a real binary form  $p$ in Theorem \ref{thm-odd}  and Theorem \ref{thm-even}  have real coefficients. 

\para

The Fast Algorithm in \cite{BFPT1} is based on Sylvester's Theorem  and the algorithm of P. Comon and B. Mourrain \cite{CM} improving some computational steps by avoiding the incremental construction that involves successively computing the kernel of  Hankel matrices. Both algorithms allow to find a WD with the minimum possible length for the given $p$ with complex coefficients. However, they achieve the result  when  a square-free binary form is found. This is no longer true in the real case. See Example \ref{ex-complex-real}.

\para

It seems to us an interesting question how to make a random choice in the most appropriate way over complex numbers.
In the real case of the WD problem,  we propose an algorithmic treatment of such a choice of linear forms. Consequently, this new algorithm  give another proof of Theorem \ref{thm-odd}  and Theorem \ref{thm-even}. Theorem \ref{thm-correctness} establishes the correctness of this new algorithm. Examples are included in Section \ref{sec-example}.

\para{}
From the computational point of view, the problem of choosing algorithmically  can be  a hard problem. It could be done using the algorithm on connected components in \cite{BPR}, (Algorithm 13.1, page 549) {but it has  high complexity}. In Algorithm \ref{alg-eleccion-Si-impar} and Algorithm   \ref{alg-eleccion-Si-par}, we present a procedure to determine an appropriate choice of real linear binary forms to solve {\it effectively} the problem $(WD)_d$ over the real field.

\para 

Here is an example to illustrate some of  the differences between the two approaches, over $\mathbb{R}$ and over $\mathbb{C}$, of the WD problem  for real binary forms.

\begin{ex}\label{ex-complex-real}
Consider the real binary  form 
\begin{equation}\label{eq-pol-ex}
 p(x,y)=a\,x^3+3b\,x^2y-3a\,xy^2-b\,y^3  \, , \textrm{ for } a\neq 0 , \, b\neq 0 \ . 
 \end{equation}

It is easy to check  it is not possible to find a  WD of length 2 \textit{over $\mathbb{R}$}. Next, consider the matrix as in the Sylvester's Algorithm, see for instance \cite{BCM},
\[\binom{-b\  \,-a\quad  b}{-a\;\quad  b\quad \;a}\]
Its kernel   is generated by $(1,0,1)$, that derives in a square-free polynomial and following the steps $4, 5$ and $6$ of the Algorithm 1 of \cite{BFPT1}, we can find a complex WD of length 2; in fact:
\[
p(x,y)=\left(\frac{a}{2}-\frac{b\,i}{2} \right)(x+i\,y)^3+ \left(\frac{a}{2}+\frac{b\,i}{2} \right)(x-i\,y)^3 \ .
\]
But, {\it over $\mathbb{R}$}, it is necessary to go to length 3. For example, we can compute a  one-parameter WD: 
\begin{equation}\label{eq-ex-con parametro}
p(x,y)=-\frac{a^2+b^2}{2s(as-b)}\,(x+s\,y)^3-\frac{a^2+b^2}{2s(as+b)}\,(x-s\,y)^3+\frac{a^3(1+s^2)}{a^2s^2-b^2}\left(x+\frac{b}{a}\,y \right),    
\end{equation}
for all  $s\neq 0 ,  \pm{b}/{a}$. 
  
  \newpage
Moreover, observe that the situations in which the shortest  decomposition is unique  are certainly scarce (see \cite{GMKP}), when we consider real WD.  It is easy to see that polynomial \eqref{eq-pol-ex} has a unique WD over $\nC$, but not over $\nR$: 
\begin{align*}
 p(x,y)= -&\frac{a^2+b^2}{s(bs^2+2as-b)}(x+sy)^3- \dfrac{(a+bs)^3}{(as-b)(bs^2+2as-b)}\left(x+\frac{-as+b}{a+bs}y\right)^3+ \\ +&\frac{(a^2+b^2)(s^2+1)}{s(as-b)}x^3\,  , \textrm{ for }  s\in \nR \setminus{}\left\{ 0 ,\, \frac{a}{b},\, \frac{b}{a},\, \frac{-a\pm \sqrt{a^2+b^2}}{b}\right\}. 
\end{align*}
that  is a different WD than \eqref{eq-ex-con parametro} since $b\not=0$ and $s\neq 0$.

\end{ex}

\para

Next, we will first establish the main algorithm, Algorithm \ref{alg-WD-new},  then we will go into details with the subroutines for choosing a finite set of appropriated parameters in  Algorithm \ref{alg-eleccion-Si-impar} and Algorithm \ref{alg-eleccion-Si-par}. In general, the linear forms  for the WD  \eqref{eq:WD_teo_impar} in Theorem \ref{thm-odd}  or \eqref{eq:WD_teo_par} in Theorem \ref{thm-even} can be chosen from a semialgebraic set, called $\mathcal{ G}$ in the previous theorems. 

\para 

\begin{lemma} [see for instance \cite{Loj}] \label{lema-roots}
Let $z^n + a_1 z^{n-1} + \cdots +a_n$ be a monic polynomial with complex coefficients. Let $\zeta_1 , \dots ,\zeta_n$ be a complete sequence of its roots. Then, we have
\begin{equation}
    \left(
    |a_i | \leq r \ , \ i =1 , \dots , n 
    \right)
    \Longrightarrow
     \left(
    |\zeta_j | \leq 2 r \ , \ j =1 , \dots , n  
    \right) .
\end{equation}
\end{lemma}

\begin{algorithm}[H]
\caption {Real Waring Decomposition (RWD))}
\label{alg-WD-new}
\SetKwInput{KwData}{Input}
\SetKwInput{KwResult}{Output}
\SetAlgoLined 
\KwData{$p(x,y)= \sum_{i=0}^{d} \binom{d}{i}c_{i}\, x^{i}\, y^{d-i}$, real binary form of  degree $d$.} 
\KwResult{an array  $ \left[\epsilon, {\bf s} ,{\bf R}, \pmb{\lambda} \right]$ such that
\begin{itemize}
    \item if $\epsilon=1$, then ${\bf s}=(s_1 ,\dots ,s_\nu , s_{\nu +1})\in \nR^{\nu +1}$ with $\nu=(d-1)/2$, ${\bf R}=(R_1 ,R_2 )$ and $\pmb{\lambda} =[ \lambda_1 ,\dots ,\lambda_d ]$ verify formula \eqref{eq:WD_teo_impar} in Theorem \ref{thm-odd}.
     \item if $\epsilon=0$, then ${\bf s}=(s,s_1 ,\dots ,s_{\nu-1} , s_{\nu })\in \nR^{\nu +1}$ with $\nu=d/2$, ${\bf R}=(R_1 ,R_2 )$ and $\pmb{\lambda} =[ \lambda_{d-1},\lambda_1 ,\dots ,\lambda_{d-2},\lambda_d ]$ verify formula \eqref{eq:WD_teo_par} in Theorem \ref{thm-even}.
\end{itemize}
 }
 Set  $\nu:=\textrm{Quotient} (d,2)$, $\epsilon:=\textrm{Remainder}(d,2)$, $c:= (c_0 , \dots ,c_d )$ . \\
     \eIf{$\epsilon=1$}
     {Call Algorithm \ref{alg-eleccion-Si-impar}, WDParameters$_{\textrm{odd}}(c)=[{\bf s} , \texttt{delta}]$ . \\ 
     \eIf{$\texttt{delta}=1$}{Define the matrix $M$ as in \eqref{eq:M_impar} by means of the elements of the list {\bf s}.\\
     Define ${\bf R}=(1,s_{\nu+1})$, where $s_{\nu+1} $ is the last element of the list {\bf s}. }{Define the matrix $M$ as $M_0$ in \eqref{eq:M_impar_0} by means of the elements of the list {\bf s}.\\
     Define ${\bf R}=(0,1)$.}
       Define the list $\pmb{\lambda}$ as the unique solution of the system $M \vec{\lambda}= \overline{c} $, where $\overline{c}= (c_d, \dots ,c_0 )^t$ and $\vec{\lambda}=(\lambda_1 ,\dots ,\lambda_d )^t$.
          }
   {Call Algorithm \ref{alg-eleccion-Si-par}, WDParameters$_{\textrm{even}}(c)=\left[{\bf s} , \texttt{delta}\right]$ .\\
        \eIf{$\texttt{delta}=1$}{Define the matrix $M$ as in \eqref{eq:M_par} by means of the elements of the list {\bf s}.\\
     Define ${\bf R}=(1,s_{\nu+1})$, where $s_{\nu+1} $ is the last element of the list {\bf s}. }{Define the matrix $M$ as $M_0$ in \eqref{eq:M_par_0} by means of the elements of the list {\bf s}.\\
     Define ${\bf R}=(0,1)$.}
       Define the list $\pmb{\lambda}$ as the unique solution of the system $M \vec{\lambda}= \overline{c} $, where $\overline{c}= (c_d, \dots ,c_0 )^t$ and $\vec{\lambda}=(\lambda_1 ,\dots ,\lambda_d )^t$.
           }
    \KwRet{ $ \left[\epsilon, {\bf s} ,{\bf R}, \pmb{\lambda} \right]$ }. \label{step:3return}
\end{algorithm}

\para 
Next we will proceed to detail the algorithms that allow us to construct an effective choice of the family of real linear forms to obtain a WD of length at most d. We  introduce some definitions in order to prove the correctness of  Algorithm \ref{alg-WD-new}.

\begin{definition}\label{def-suitable-odd}
Let $p(x,y)$ be a real binary form of odd degree $d=2\nu +1$. Let ${\bf s}=(s_1 ,\dots ,s_\nu ,s_{\nu+1})=(s', s_{\nu+1})$ be a vector in $ \nR^{\nu+1}$. We say that  ${\bf s}$ is {\it a vector of suitable parameters for a WD of $p$} if it satisfies the following conditions:
\begin{itemize}
    \item for $i=1, \dots ,\nu$, we have $s_i \not=0$,
    \item $s_i\not= \pm s_j$, for $i\neq j$,
    \item $s_{\nu +1}=-\dfrac {\Delta_{d-1}(s')}{\Delta_d(s')}$ whenever $\Delta_d(s')\not=0$ and $s_{\nu +1}=0$ otherwise,  with $\Delta_i (s')$ stands for the evaluation at $s'$ of the  polynomial $\Delta_i$ defined in \eqref{eq-det-odd}.
   \end{itemize}
  
\end{definition}

One can easily verify the following result.

\begin{lemma}\label{lema-G-odd}
Let $p$ be a real binary form of odd degree $d=2\nu +1$.  Let ${\bf s }=(\textsc{s},s_{\nu +1})$ be a vector in $\mathbb{R}^{\nu+1}$ of suitable parameters for a WD of $p$. Then the point $\textsc {s}$  belongs to the semialgebraic set $\cG$ defined in \eqref{conj-G-odd}.
\end{lemma}

\begin{definition}\label{def-suitable-even}
Let $p(x,y)$ be a real binary form of even degree $d=2\nu $. Let ${\bf s}=(s^*,s_1 ,\dots ,s_{\nu-1} ,s_{\nu+1})=(s', s_{\nu+1})$ be a vector in $ \nR^{\nu+1}$. We say that  ${\bf s}$ is {\it a vector of suitable parameters for a WD of $p$} if it satisfies  the following conditions:
\begin{itemize}
    \item for $i=1, \dots ,\nu-1$, we have $s_i \not=0$,
    
    \item $s_{\nu +1}=-\dfrac {\Delta_{d-1}(s')}{\Delta_d(s')}-s^*$ whenever $\Delta_d(s')\not=0$ and $s_{\nu +1}=0$ otherwise,  with $\Delta_i (s')$ stands for the evaluation at $s'$ of the  polynomial $\Delta_i$ defined in \eqref{eq-det-even}.
    \item $s^*\neq 0$ and $ s_i \neq \pm s^*$ for  $i =1, \dots ,\nu-1 , \nu+1$, and 
    \item $s_i\not= \pm s_j$, for all $i\neq j$.
   \end{itemize}
  \end{definition}

The following result follows from the definition of \eqref{conj-G-even}.

\begin{lemma}\label{lema-G-even}
Let $p$ be a real binary form of odd degree $d=2\nu +1$.  Let ${\bf s }=(\textsc{s},s_{\nu +1})$ be a vector in $\mathbb{R}^{\nu+1}$ of suitable parameters for a WD of $p$. Then the point $\textsc {s}$  belongs to the semialgebraic set $\cG$ defined in  \eqref{conj-G-even}.
\end{lemma}

\begin{algorithm}[H]
\caption {WDParameters\_odd}
\label{alg-eleccion-Si-impar}
\SetKwInput{KwData}{Input}
\SetKwInput{KwResult}{Output}
\SetAlgoLined 
\KwData{$p(x,y)= \sum_{i=0}^{d} \binom{d}{i}c_{i}\, x^{i}\, y^{d-i}$, real binary form of odd degree $d=2\nu+1$.} 
\KwResult{an array $\left[{\bf s} , \texttt{delta}\right]$ where ${\bf s}\in\nR^{\nu+1}$   is a vector   of suitable parameters for the WD of $p$ and \texttt{delta} is a switch that returns $0$ if  or $1$ otherwise. }
\medskip
    Set  $\nu:=\ (d-1)/{2}$ and \ \texttt{delta}:=0\ .
    Define $s_i:=i$, for $i=1 ,\dots ,\nu -1$, and $s'=(s_1 ,\dots ,s_{\nu -1})$.\\
       Compute the polynomials $\Delta_d (S_1 ,\dots ,S_\nu )$ and $\Delta_{d-1}(S_1 ,\dots ,S_\nu )$ from Form. \eqref{eq-det-odd}.
       \\
    Compute $ \pmb{\Delta_{d}}(S_\nu ) = \Delta_d (s', S_\nu )$ and  $ \pmb{\Delta_{d-1}}(S_\nu ) = \Delta_{d-1} (s', S_\nu )$
    .\\
   \eIf{$\pmb{\Delta_{d}}(S_\nu )\equiv 0$}{Define $\delta_{d-1}$ as the maximum of the absolute value of the coefficients of the polynomial $\pmb{\Delta_{d-1}}(S_\nu ) $
    \\ 
    Compute $m=\max \{ \nu-1, \delta_{d-1}\}$
    \\
    Define $s_\nu =2m+1$.  \\
    \label{step:s_nu_odd_0}
    Set ${\bf s}:=(s_1,\ldots,s_\nu,0)$ and go to step \ref{step:return}. \label{step:caso11}}
   {\texttt{delta}$\,\leftarrow 1$\\
   Define $\delta_{d}$ as the maximum of the absolute value of the coefficients of the polynomial $\pmb{\Delta_{d}}(S_\nu ) $.
     \\
    Define $\delta_{d-1,0}^{+}$ as the maximum of the absolute value of the coefficients of the polynomial $\pmb{\Delta_{d-1}}(S_\nu ) -S_\nu \pmb{\Delta_{d}}(S_\nu )$.
    \\
    Define $\delta_{d-1,0}^{-}$ as the maximum of the absolute value of the coefficients of the polynomial $\pmb{\Delta_{d-1}}(S_\nu ) +S_\nu \pmb{\Delta_{d}}(S_\nu )$.\\
   \For{ $i\leftarrow  1$ \KwTo  $\nu-1$}{
     \begin{itemize}
         \item Define $\delta_{d-1,i}^{+}$ as the maximum of the absolute value of the coefficients of the polynomial $\pmb{\Delta_{d-1}}(S_\nu ) -s_i \pmb{\Delta_{d}}(S_\nu )$.
         \item Define $\delta_{d-1,i}^{-}$ as the maximum of the absolute value of the coefficients of the polynomial $\pmb{\Delta_{d-1}}(S_\nu ) +s_i \pmb{\Delta_{d}}(S_\nu )$.
     \end{itemize}}
Compute $m=\max \{ \nu-1, \delta_d ,\delta_{d-1,0}^{\pm},\dots ,  \delta_{d-1, \nu-1}^{\pm}\}$.\\
   Define $s_\nu =2m+1$.  \\
   \label{step:s_nu_odd}
   Compute $R= -\dfrac{\pmb{\Delta_{d-1}}(s_\nu )}{\pmb{\Delta_{d}}(s_\nu )}$.\\ \label{step:R_odd}
   Set ${\bf s}:=(s_1,\ldots,s_\nu,R)$ and go to step \ref{step:return}.\label{step:caso14}
}
\KwRet{   $ \left[{\bf s} , \texttt{delta}\right]$}.\\ \label{step:return}
\end{algorithm}

\para 

The following proposition guarantees the correctness of Algorithm \ref{alg-eleccion-Si-impar}.

\begin{proposition}\label{correct-odd} Let $p(x,y)$ be a real binary form of odd degree $d=2\nu +1$. The output list $ \left[{\bf s} , \texttt{delta}\right]$ from Algorithm \ref{alg-eleccion-Si-impar} applied to $p$ defines the vector  ${\bf s}=(s_1,\ldots,s_\nu , s_{\nu +1})$ of suitable parameters for a WD of $p$.

Moreover, if $\texttt{delta}=1$ (respectively $\texttt{delta}=0$) then the linear system \eqref{eq:sistema1}  (respec\-tively \eqref{eq:d-impar-0}) using the vector ${\bf s}$ has  unique solution.
\end{proposition}

In the following proof we use the notations in Algorithm \ref{alg-eleccion-Si-impar}.
\newpage

\begin{proof} Let $ \left[{\bf s} , \texttt{delta}\right]$ be the returned array of Algorithm \ref{alg-eleccion-Si-impar}.
First assume    $\texttt{delta}=1$. Then ${\bf s}:=(s_1,\ldots,s_\nu,R)$ with $R= -\dfrac{\pmb{\Delta_{d-1}}(s_\nu )}{\pmb{\Delta_{d}}(s_\nu )} $. For $i=1,\dots , \nu-1$, consider the rational functions:
\begin{equation}
\Psi_{i,+} (S_\nu )=- \dfrac{\pmb{\Delta_{d-1}}(S_\nu )}{\pmb{\Delta_{d}}(S_\nu )}-s_i \ , \quad
   \Psi_{i,-} (S_\nu )=- \dfrac{\pmb{\Delta_{d-1}}(S_\nu )}{\pmb{\Delta_{d}}(S_\nu )}+s_i \ ,    
\end{equation}
and also
\begin{equation}
\Psi_{0,+} (S_\nu )=- \dfrac{\pmb{\Delta_{d-1}}(S_\nu )}{\pmb{\Delta_{d}}(S_\nu )}-S_\nu \ , \quad
   \Psi_{0,-} (S_\nu )=- \dfrac{\pmb{\Delta_{d-1}}(S_\nu )}{\pmb{\Delta_{d}}(S_\nu )}+S_\nu \ .    
\end{equation}
Then, for $i=0,\dots , \nu$, we obtain that
\begin{equation}
 \Psi_{i,+} (s_\nu )=- \dfrac{\pmb{\Delta_{d-1}}(s_\nu )+s_i \pmb{\Delta_{d}}(s_\nu )}{\pmb{\Delta_{d}}(s_\nu )} \ , \quad
   \Psi_{i,-} (s_\nu )=- \dfrac{\pmb{\Delta_{d-1}}(s_\nu )-s_i \pmb{\Delta_{d}}(s_\nu )}{\pmb{\Delta_{d}}(s_\nu )}    
\end{equation}
are nonzero real numbers, because of Lemma \ref{lema-roots} and the choice made in Step \ref{step:s_nu_odd} for $s_{\nu}$. Moreover, $R\neq \pm s_i$, for $i=1, \dots , \nu$. 

As a consequence of the previous construction, we obtain 
 the vector $\textbf{s}=(s_1,\ldots,s_\nu,s_{\nu+1} )$ satisfying the requirements of definition \ref{def-suitable-odd}, since $s_i\neq \pm s_\nu$, for $i\neq\nu$, by choice in Step \ref{step:s_nu_odd}, and $\Psi_{i,\pm} (s_\nu )\neq 0$ for all $i$, {and also $s_{\nu+1}=R$ satisfies the third condition according to the Step \ref{step:R_odd}}.

Moreover the linear system \eqref{eq:sistema1} has a unique solution, since $s_i\neq \pm s_\nu$, for $i=1,\dots ,\nu-1$, and $R\neq \pm s_i$, by choice in Step \ref{step:s_nu_odd}.

Now assume    $\texttt{delta}=0$. Then define $s_\nu $ as in Step \ref{step:s_nu_odd_0}. Hence  $s_i\neq \pm s_j$, for all $i\neq j$, so the linear system \eqref{eq:d-impar-0} has a unique solution and the point $\textsc{s}=(s_1,\ldots,s_\nu )$ is in the semialgebraic set $\cG$ defined in \eqref{conj-G-odd}, in the proof of Theorem \ref{thm-odd}.

As a result of all the above we have proved the required statement. 
\end{proof}

\begin{obs}
Observe that, as a consequence of the previous proof, we obtain that the point $\textsc{s}=(s_1,\ldots,s_\nu )$ is in the semialgebraic set $\cG$ defined in \eqref{conj-G-odd}, in the proof of Theorem \ref{thm-odd}, since $s_i\neq \pm s_\nu$, for $i\neq\nu$, by choice in Step \ref{step:s_nu_odd}, and $\Psi_{i,\pm} (s_\nu )\!\neq\! 0$ for all $i$.
\end{obs}

\para

\begin{algorithm}[H]
\caption {WDParameters\_even}
\label{alg-eleccion-Si-par}
\SetKwInput{KwData}{Input}
\SetKwInput{KwResult}{Output}
\SetAlgoLined 
\KwData{$p(x,y)= \sum_{i=0}^{d} \binom{d}{i}c_{i}\, x^{i}\, y^{d-i}$, real binary form of even degree $d$.} 
\KwResult{a vector  $(s,\,s_i),\,i=1,\ldots,\lfloor{\frac{d}{2}}\rfloor-1$, of suitable parameters for the WD of $p$ and a switch, \texttt{delta}, that returns zero when $\Delta_d=0$ or $1$ otherwise. }

\para 

    Set  $\nu:={d}/{2}$, \ \text{and } \texttt{delta}:=0. \\
     Define $s_i:=i$, for $i=1 ,\dots ,\nu-1$, and $s'':=(s_1 ,\dots ,s_{\nu -1})$. \\
       Compute the polynomials $\Delta_d (S, S_1 ,\dots ,S_{\nu-1} )$ and $\Delta_{d-1}(S, S_1 ,\dots ,S_{\nu-1})$ from formula \eqref{eq-det-even}
       .\\
    Compute $ \pmb{\Delta_{d}}(S ) = \Delta_d (S, s'' )$ and  $ \pmb{\Delta_{d-1}}(S ) = \Delta_{d-1} (S, s'' )$.
    \\
   \eIf{$\pmb{\Delta_{d}}(S )\equiv 0$}{Define $\delta_{d-1}$ as the maximum of the absolute value of the coefficients of the polynomial $\pmb{\Delta_{d-1}}(S ) $.
    \\ 
    Compute $m=\max \{ \nu-1, \delta_{d-1}\}$.
    \\
    Define $s =2m+1$.  \\ \label{step:s_nu_even_0}
    Set ${\bf s}:=(s,s_1,\ldots,s_{\nu-1},0)$ and go to step \ref{step:return_par}. \label{step:caso21}}{
    \texttt{delta}$\,\leftarrow 1$\\
   Define $\delta_{d}$ as the maximum of the absolute value of the coefficients of the polynomial $\pmb{\Delta_{d}}(S ) $.
     \\
       Define $\delta_{d-1,0}^{-}$ as the maximum of the absolute value of the coefficients of the polynomial $\pmb{\Delta_{d-1}}(S ) +2 S \pmb{\Delta_{d}}(S )$. \\
   \For{ $i\leftarrow  1$ \KwTo  $\nu-1$}{
     \begin{itemize}
         \item Define $\delta_{d-1,i}^{+}$ as the maximum of the absolute value of the coefficients of the polynomial $\pmb{\Delta_{d-1}}(S ) -(S+s_i) \pmb{\Delta_{d}}(S )$.
         \item Define $\delta_{d-1,i}^{-}$ as the maximum of the absolute value of the coefficients of the polynomial $\pmb{\Delta_{d-1}}(S ) +(S- s_i ) \pmb{\Delta_{d}}(S )$.
     \end{itemize}}
Compute $m=\max \{ \nu-1, \delta_d ,\delta_{d-1,0}^{-},\delta_{d-1,1}^{\pm},\dots ,  \delta_{d-1, \nu-1}^{\pm} \}$.
\\
   Define $s =2m+1$.  \\ \label{step:s_even}
   Compute $R= -\dfrac{\pmb{\Delta_{d-1}}(s )}{\pmb{\Delta_{d}}(s )}-s$. \\
   \label{step:R_even}
   Set ${\bf s}:=(s, s_1,\ldots,s_{\nu-1},R)$ and go to step \ref{step:return_par}.\label{step:caso24}
}
\KwRet{ $ \left[{\bf s} , \texttt{delta}\right]$  }. \\ \label{step:return_par}
\end{algorithm}
\para 

The following proposition guarantees the correctness of Algorithm \ref{alg-eleccion-Si-par}.

\begin{proposition}\label{correct-even} Let $p(x,y)$ be a real binary form of even degree $d=2\nu$. The output list $ \left[{\bf s} , \texttt{delta}\right]$ from Algorithm \ref{alg-eleccion-Si-par} applied to $p$ defines the vector  $\textbf{s}=(s,s_1,\ldots,s_{\nu-1} , s_{\nu +1})$ of suitable parameters for a WD of $p$.

Moreover, if $\texttt{delta}=1$ (respectively $\texttt{delta}=0$) then the linear system \eqref{eq:sistema2}  (respectively \eqref{eq:sistema-par-0}) using the vector ${\bf s}$ has  unique solution.
\end{proposition}

In the following proof we use the notations in Algorithm \ref{alg-eleccion-Si-par}.
\newpage

\begin{proof}

Let $ \left[{\bf s} , \texttt{delta}\right]$  be the returned array of Algorithm \ref{alg-eleccion-Si-par}. First, assume    $\texttt{delta}=1$. Then ${\bf s}:=(s, s_1,\ldots,s_{\nu-1},R)$ with $R= -\dfrac{\pmb{\Delta_{d-1}}(s )}{\pmb{\Delta_{d}}(s )}-s$. For $i=1,\dots , \nu-1$, consider the rational functions:
\begin{equation}
\Psi_{i,+} (S )=- \dfrac{\pmb{\Delta_{d-1}}(S )}{\pmb{\Delta_{d}}(S )}-s_i -S\ , \quad
   \Psi_{i,-} (S )=- \dfrac{\pmb{\Delta_{d-1}}(S )}{\pmb{\Delta_{d}}(S )}+s_i -S\ ,    
\end{equation}
and also
\begin{equation}
\Psi_{0,+} (S )=- \dfrac{\pmb{\Delta_{d-1}}(S )}{\pmb{\Delta_{d}}(S )}-2S   \ .    
\end{equation}
Then, for $i=1,\dots , \nu-1$, we obtain that
\begin{equation}
 \Psi_{i,+} (s )=- \dfrac{\pmb{\Delta_{d-1}}(s )+(s+s_i )\pmb{\Delta_{d}}(s )}{\pmb{\Delta_{d}}(s )} \ , \quad
   \Psi_{i,-} (s )=- \dfrac{\pmb{\Delta_{d-1}}(s )+(s-s_i )\pmb{\Delta_{d}}(s )}{\pmb{\Delta_{d}}(s )} \ ,   
\end{equation}
\[
\Psi_{0,+} (s )=- \dfrac{\pmb{\Delta_{d-1}}(s )-2s\pmb{\Delta_{d}}(s )}{\pmb{\Delta_{d}}(s )} 
\]
are nonzero real numbers, because of Lemma \ref{lema-roots} and the choice made in Step \ref{step:s_even} for $s$. Moreover $R\neq \pm s_i$ for $i=1, \dots , \nu-1$ and also $R\neq s$. 

Then we obtain the vector $\,{\textbf{s} }=(s,s_1,\ldots,s_{\nu-1}, s_{\nu+1} )$ of suitable parameters,  since $s_i\neq \pm s$, for $i<\nu$, by choice in Step \ref{step:s_even}, and $\Psi_{i,\pm} (s )\neq 0$ for all $i$, {and also $s_{\nu+1}=R$ satisfies the third condition according to the Step \ref{step:R_even}}.

Moreover the linear system \eqref{eq:sistema2} has a unique solution, since $s_i\neq \pm s$, for $i=1,\dots ,\nu-1$, and $R\neq \pm s_i$ by choice in Step \ref{step:s_even}.

Now assume    $\texttt{delta}=0$. Then define $s $ as in Step \ref{step:s_nu_even_0}. Hence  $s_i\neq \pm s_j$, for all $i\neq j$, so the linear system \eqref{eq:sistema-par-0} has a unique solution and the point $\textsc{s}=(s, s_1,\ldots,s_{\nu-1} )$ is in the semialgebraic set $\cG$ defined in \eqref{conj-G-even}, in the proof of Theorem \ref{thm-even}.

As a result of all the above we have proved the required statement. 
\end{proof}

\begin{obs}
As a consequence of the previous construction, we obtain that
  ${\textsc{s} }=(s,s_1,\ldots,s_{\nu-1} )$ is in the semialgebraic set $\cG$ defined in \eqref{conj-G-even},  since $s_i\neq \pm s$, for $i<\nu$, by choice in Step \ref{step:s_even}, and $\Psi_{i,\pm} (s )\!\neq\! 0$ for all $i$.
\end{obs}

\para 

The following result guarantees the correctness of the  Algorithm  \ref{alg-WD-new}.

\begin{theorem}\label{thm-correctness} Let $p(x,y)$ be a real binary form of  degree $d$.
 The output list $ \left[\epsilon, {\bf s} ,{\bf R}, \mathbf{ \lambda} \right]$ from Algorithm \ref{alg-WD-new} applied to $p$ verifies formula \eqref{eq:WD_teo_impar} in Theorem \ref{thm-odd} if $\epsilon=1$, or formula \eqref{eq:WD_teo_par} in Theorem \ref{thm-even} if $\epsilon=0$.
\end{theorem} 

 In the following proof we use the notations in Algorithm \ref{alg-WD-new}.

\begin{proof} First assume $d$ is an odd number. From Proposition \ref{correct-odd},  we obtain a point ${\bf s}$ in $\cG$, defined in \eqref{conj-G-odd}, such that the corresponding linear system \eqref{eq:sistema1} or \eqref{eq:d-impar-0}, according to \texttt{delta}, has an unique solution. Analogously, for even degree $d$, we consider the linear system \eqref{eq:sistema2} or \eqref{eq:sistema-par-0}. By Proposition \ref{correct-even}, it has a unique solution, since $\textbf{s}=(\textsc{s},R)$ is a vector of suitable parameters for $p$ and $\textsc{s}$ is in the set $\cG$ from \eqref{conj-G-even}. Hence Algorithm \ref{alg-WD-new} computes a WD of length at most $d$ for any real binary form.
\end{proof}

\para 
\begin{obs} Observe that Algorithm \ref{alg-WD-new} computes a WD of length at most $d$ for any real binary form of any degree $d$. Moreover, Theorem \ref{thm-odd} and Theorem \ref{thm-even} also describe a semialgebraic set where we can choose the appropriate parameters to construct such a WD for each real binary form $p$.
\end{obs}

Next we present  an example to illustrate our results.

\para 

\noindent \textbf{\textsc{A detailed example}}
\para
\label{sec-example}

 \label{ex_parameters_5}
Next consider    the  binary form 
\begin{equation}
p_a (x,y)=x^5+10x^4y+10x^3y^2+10x^2y^3+5axy^4+y^5\,  , \ \textrm{ for } a= 2 , 1.
\end{equation}
Running Algorithm \ref{alg-WD-new} we obtain:
\begin{itemize}
    \item For $a=2$:
    \begin{enumerate}[1.]
    \item  Set $\nu:=2$, $\epsilon:=1$ and $\sc{c}:=(1,2,1,1,2,1)$.
    \item Call the Algorithm \ref{alg-eleccion-Si-impar} that gives us the list $[{\bf s},\texttt{delta} ]=[(1,25,625),1]$.
   
        \item From Algorithm \ref{alg-WD-new}, Step 5, we  define the linear system
        $$
        \begin{pmatrix}
        1&1&1&1&1\\
        1&-1&25&-25&625\\
        1&1&625&625&390625\\
        1&-1&15625&-15625&244140625\\
        1&1&390625&390625&152587890625\\
        1&-1&9765625&-9765625&95367431640625\\
        \end{pmatrix}
        \begin{pmatrix}
        \lambda_1\\ \lambda_2\\ \lambda_3\\
        \lambda_4\\ \lambda_5\\
        \end{pmatrix}
        =
    \begin{pmatrix}
        1\\2\\1\\1\\2\\ 1   
    \end{pmatrix}
        $$
          
        \item Define  ${\bf{R}}=(1,625)$.
        \item RETURN 
        \[
       \textstyle{ \left[1,(1,25,625),(1,625),\left(\frac{1168753}{778752},-\frac{130417}{260416},-\frac{601}{18720000},\frac{217}{6760000},\frac{1}{152343360000}\right)\right]}
        \] 
 Therefore,
 \[
 \begin{array}{rl}
       p_2 (x,y) =& \dfrac{1168753}{778752} (x + y)^5- \dfrac{130417}{260416} (x - y)^5 + \dfrac{-601}{18720000} (x + 25 y)^5 +\\
       & \ \\
       &\dfrac{217}{6760000} (x - 25 y)^5 + \dfrac{1}{152343360000}(x+625 y)^5 \ .
      \end{array}
\]
            
\end{enumerate}
\item  For $a=1$:
    \begin{enumerate}[1.]
    \item  Set $\nu:=2$, $\epsilon:=1$ and $\sc{c}:=(1,1,1,1,2,1)$.
    \item Call the Algorithm \ref{alg-eleccion-Si-impar} that gives us the list $[{\bf s},\texttt{delta} ]=[(1,17,0),0]$.
    \item From Algorithm \ref{alg-WD-new}, Step 8, we  define the linear system
    $$
        \begin{pmatrix}
        1&1&1&1&0\\
        1&-1&17&-17&0\\
        1&1&289&289&0\\
        1&-1&4913&-4913&0\\
        1&1&83521&83521&0\\
        1&-1&1419857&-1419857&1\\
        \end{pmatrix}
        \begin{pmatrix}
        \lambda_1\\ \lambda_2\\ \lambda_3\\
        \lambda_4\\ \lambda_5\\
        \end{pmatrix}
        =
    \begin{pmatrix}
        1\\2\\1\\1\\1\\ 1   
    \end{pmatrix}
        $$
        \vspace{-4mm}
        \item Define  ${\bf R}=(0,1)$.
        \item RETURN \[\left[1,(1,17,0),(0,1),\left(\dfrac{865}{576},\dfrac{-289}{576},\dfrac{-1}{9792},\dfrac{1}{9792},289\right)\right]
        \]
     \end{enumerate} 
           
\end{itemize}        
        \[ \textrm{And so, }
        p_1 (x,y)=\frac{865}{576} (x + y)^5- \frac{289}{576} (x - y)^5 -\frac{ 1}{9792} (x + 17 y)^5 + \frac{ 1}{9792} (x - 17 y)^5 + 289 y^5 \ .
        \]

Observe that we have chosen quite restrictively the type of linear forms for the  Waring decompositions, since we take $s_i$ and $-s_i$ in our constructions. Sometimes, under this condition, it is not possible to find a shorter WD, but,  removing this restriction, we can find an improved one. For instance, 
$$
p_1(x,y)=\frac{1}{11}\left[ \ \lambda_1(x+\alpha_1 y)^5+\overline{\lambda}_1 (x+\overline{\alpha}_1 y)^5+\lambda_3(x+\alpha_3 y)^5+\overline{\lambda}_3 (x+\overline{\alpha}_3  y)^5 \ \right],
$$
with $\lambda_1=-6+\frac{14}{3}\sqrt{3} $, $\alpha_1=\frac{1}{2}(1+\sqrt{3}) $, $\lambda_3=-\frac{23}{2}+\frac{51}{10}\sqrt{5} $ and $\alpha_3=\frac{1}{2}(-1+\sqrt{5}) $, using the notation $\overline{w}=m-n\sqrt{p}$, when  $ w=m+n\sqrt{p}$.

\begin{obs}
For $a\neq 1$,  we receive $[{\bf s},\texttt{delta} ]=\left[\left(1,s_2,\frac{s_2^2}{a-1}\right),1\right]$ (see  Algorithm \ref{alg-eleccion-Si-impar}). Next, we must define the matrix $M$ as in \eqref{eq:M_impar} by means of the elements of the list {\bf s}, and then solve the corresponding linear system \eqref{eq:sistema1}. All these processes can be done symbolically.
\end{obs}

\begin{obs}
Note that for fixed  value of $a$, we can determine a real  value $s_2$ such that $\{ 1 \} \times [s_2 , \infty ) \subset \cG$, with $\cG$ the semialgebraic set defined in \eqref{conj-G-odd} depending on $p_a$. For instance, for $a=1$ we have $\{ 1 \} \times [17,\infty)\subset \mathcal{G}$, and $\{ 1 \} \times [25,\infty)\subset \mathcal{G}$ for $a=2$.

This gives, in general, a way to find the values of $\textsc{s}\in\mathcal{G}$. Certainly, there are smaller values of $s_2$ that are also in $\mathcal{G}$. For instance, the vector $\textbf{s}=(\textsc{s},R)=(1,2,\frac{4}{a-1})$ is also a vector of suitable parameters, but our  bound  guarantees  $\{ 1 \} \times [s_2 ,\infty)\subset\mathcal{G}$ for those $s_2$, and accordingly the associated linear system for these values has a unique solution. 

In fact, running Algorithm \ref{alg-WD-new} over "a symbolic ring", we obtain the one-parameter WD:
\begin{align*}
p_2(x,y)=&\frac{3s_2^4-5s_2^2+3}{2(s_2^2-1)^2}(x+y)^5- \frac{s_2^4+s_2^2+1}{2(s_2^2+1)(s_2^2-1)}(x-y)^5-
\\
&\frac{s_2^2-s_2+1}{2s_2^2(s_2-1)^2(s_2+1)}(x+s_2\,y)^5+
\frac{s_2^2+s_2+1}{2s_2^2(s_2-1)(s_2+1)^2}(x-s_2\,y)^5 +\\
&\frac{1}{s_2^2(s_2^2+1)(s_2-1)^2(s_2+1)^2}(x+s_2^2\,y)^5 \ ,    
\end{align*}
for any real number $s_2$ such that $ s_2  \neq 0 , \pm 1$. 
\end{obs}


\bibliographystyle{siam}
\bibliography{bibliography.bib}


\end{document}